\documentclass[11pt]{article}

\usepackage{cite}
\usepackage{subfigure}
\usepackage{graphicx, color, graphpap}
\usepackage[normalem]{ulem}
\usepackage{amsmath,amssymb,amsthm}
\usepackage{ifthen}
\usepackage{mathrsfs}
\usepackage{mathtools}
\usepackage{enumitem}
\usepackage{lineno}
\usepackage{float}
\usepackage{fancyhdr}
\usepackage[us,12hr]{datetime} 
\usepackage{exscale}
\usepackage{tabularx}
\usepackage{latexsym}
\usepackage{fullpage}       
\usepackage{verbatim}
\usepackage{multirow}
\usepackage{overpic}
\usepackage{comment} % writing comments
\usepackage{hyperref}
\usepackage{blkarray}
\usepackage{adjustbox}
\usepackage{tcolorbox}

\usepackage{pgfplots}
\usepackage{marvosym}
\usepackage[noabbrev]{cleveref}
\usepackage{tikz}

\usepackage{algorithm}
\usepackage{algpseudocode}

\pgfplotsset{compat=1.11}
\usetikzlibrary{arrows, arrows.meta}
\usetikzlibrary{shapes.geometric, arrows.meta, positioning}

\tikzstyle{startstop} = [rectangle, rounded corners, minimum width=3.5cm, minimum height=1cm,text centered, draw=black, fill=blue!10]
\tikzstyle{process} = [rectangle, minimum width=4cm, minimum height=1cm, text centered, draw=black, fill=green!10]
\tikzstyle{io} = [trapezium, trapezium left angle=70, trapezium right angle=110, minimum width=3.5cm, minimum height=1cm, text centered, draw=black, fill=orange!15]
\tikzstyle{arrow} = [thick,->,>=stealth]

\definecolor{methD}{RGB}{230,120,50}
\definecolor{methDD}{RGB}{50,140,90}
\definecolor{methSDD}{RGB}{50,90,170}
\definecolor{methPSD}{RGB}{180,40,40}
\tikzset{
	visbox/.style={rounded corners=8pt, line width=1.2pt, align=center, inner sep=8pt},
	visnestouter/.style={visbox, minimum width=9.6cm, minimum height=4.4cm},
	visnest/.style={visbox, minimum width=7.6cm, minimum height=3.25cm},
	visnestmid/.style={visbox, minimum width=5.7cm, minimum height=2.2cm},
	visnestinn/.style={visbox, minimum width=3.4cm, minimum height=1.05cm,
		inner sep=3pt, font=\scriptsize\bfseries, align=center, text=methD},
	visnestpos/.style={anchor=south east},
	visnestinset/.style={xshift=-5pt, yshift=5pt},
}

\newcommand{\FW}{\mathbb{FW}}

\newcommand{\Diag}{\text{Diag}}

\DeclareMathOperator{\range}{range}

\newcommand{\supp}{\text{supp}}

\newcommand{\A}{\mathcal{A}}

\newcommand{\aff}{\text{aff}}

\newtheorem{prop}{Proposition}
\newtheorem{lem}{Lemma}[section]
\newtheorem{thm}{Theorem}[section]
\newtheorem{cor}{Corollary}[section]

\newtheorem{ex}{Example}[section]

\crefname{thm}{Theorem}{Theorems}
\Crefname{thm}{Theorem}{Theorems}
\crefname{problem}{Problem}{Theorems}
\Crefname{problem}{Problem}{Theorems}
\Crefname{assump}{Assumption}{Theorems}
\crefname{assump}{Assumption}{Theorems}
\crefname{assumption}{Assumption}{Assumptions}
\Crefname{assumption}{Assumption}{Assumptions}
\crefname{conjecture}{Conjecture}{Theorems}
\Crefname{conjecture}{Conjecture}{Theorems}
\crefname{prop}{Proposition}{Propositions}
\Crefname{prop}{Proposition}{Propositions}
\crefname{cor}{Corollary}{Corollaries}
\Crefname{cor}{Corollary}{Corollaries}
\crefname{lem}{Lemma}{Lemmas}
\Crefname{lem}{Lemma}{Lemmas}
\theoremstyle{definition}
\crefname{conj}{Conjecture}{Conjectures}
\Crefname{conj}{Conjecture}{Conjectures}
\crefname{remark}{Remark}{Remarks}
\Crefname{remark}{Remark}{Remarks}
\crefname{rmk}{Remark}{Remarks}
\Crefname{rmk}{Remark}{Remarks}
\crefname{example}{Example}{Examples}
\Crefname{example}{Example}{Examples}
\crefname{align}{}{}
\Crefname{align}{}{}
\crefname{equation}{}{}
\Crefname{equation}{}{}

\def\eqref#1{{\normalfont(\ref{#1})}}
\usepackage{authblk}

\author[1]{Hao Hu\thanks{Corresponding author. Email: \url{hhu2@clemson.edu}.}}
\author[1]{Blake Smith\thanks{Email: \url{bjs5@g.clemson.edu}.}}
\author[1]{Mingming Xu\thanks{Email: \url{mingmix@g.clemson.edu}.}}
\affil[1]{School of Mathematical and Statistical Sciences, Clemson University, Clemson, USA}

\begin{document}

\title{Face Identification via Polyhedral Relaxations for Binary Programs: A Comparison with Factor-Width Partial Facial Reduction}

\break
\date{\currenttime, \today}
\maketitle

\medskip

\begin{abstract}
\begingroup
Facial reduction (FR) is a preprocessing technique used to restore Slater's
condition in semidefinite programming (SDP), but each step generally requires
solving an auxiliary SDP.  Factor-width-\(k\) partial FR restricts the class of
exposing matrices.  Increasing \(k\) enlarges this class, but, for \(k\geq3\),
the auxiliary problem remains an SDP involving positive semidefinite blocks of
order \(k\).  For SDP relaxations of binary programs, we replace this auxiliary
SDP by an LP-based face-identification method.  We construct a polyhedron in
subset-indexed moment variables that has polynomial size for fixed \(k\).  The
resulting face contains every feasible binary lift and is contained in every
face exposed by a factor-width-\(k\) exposing matrix from the same auxiliary
system.  Thus, linear programming identifies a face no larger than those
obtained from the factor-width-\(k\) auxiliary SDP while preserving all
feasible binary points.  To iterate the construction, we represent each
identified face by linear equations in the original matrix variable, thereby
retaining the moment-matrix indexing and the factor-width comparison at every
iteration.

\endgroup
\end{abstract}


\section{Introduction}

\begingroup
Semidefinite programming (SDP) relaxations provide strong bounds for many
binary optimization problems
\cite{shor1987quadratic,lovasz1991cones,goemans1995improved,laurent2003comparison}.
They lift the problem to a matrix space while retaining the quadratic
identities satisfied by binary variables.
Higher-order constructions based on the reformulation-linearization technique
(RLT)---equivalently, the Sherali--Adams hierarchy in this setting---and
moment methods exploit the same binary structure
\cite{sherali1990hierarchy,sherali2013reformulation,laurent2012semidefinite}.

SDP relaxations may fail Slater's condition when their feasible matrices are
confined to a proper face of the positive semidefinite cone.  In relaxations
of binary programs, combinatorial equalities and their lifted consequences
can force all feasible matrices to share a nontrivial nullspace.  Facial
reduction (FR) addresses this failure of strict feasibility by identifying a
smaller face that contains the feasible set
\cite{borwein1981regularizing,borwein1981facial,waki2013facial}.  Each step,
however, ordinarily requires solving an auxiliary SDP to obtain an exposing
matrix.  Partial facial reduction reduces this cost by restricting the search
for an exposing matrix to an inner approximation of the positive semidefinite
cone \cite{permenter2018simplified}.  Factor-width cones provide one such
hierarchy.  Factor widths one and two lead, respectively, to a linear program
(LP) and a second-order cone program (SOCP), although the resulting restricted
auxiliary problems may identify only limited facial structure; relevant
structural characterizations and limitations are given in
\cite{hu2026facestructure}.  Increasing the factor width enlarges the class
of available exposing matrices.  For \(k\geq3\),
however, the auxiliary problem is again an SDP, now involving \(k\times k\)
positive semidefinite constraints \cite{boman2005factor}.

This trade-off motivates a direct use of the binary structure.
We construct a polyhedron from local RLT bound-factor inequalities in
subset-indexed moment variables.  By the standard subset-indexed
moment-matrix factorization, these inequalities encode positive
semidefiniteness of the corresponding local moment matrices
\cite[Sections~3.1--3.2]{laurent2003comparison} and
\cite{sherali1990hierarchy}.  For
fixed \(k\), the polyhedron has polynomial size, and the resulting face can be
identified using linear programming.  We prove that this face contains all
feasible binary lifts and is contained in every face exposed by a
factor-width-\(k\) exposing matrix from the auxiliary system.  Thus, the face
comparison is obtained by LP rather than by solving the factor-width auxiliary
SDP.

To state the comparison precisely, denote the resulting face by
\(F_{\mathrm{aff}}\), and let \(P_{\mathrm{lift}}\) be the set of rank-one
lifts of the feasible binary points.  Although \(F_{\mathrm{aff}}\) contains
\(P_{\mathrm{lift}}\), it may
exclude fractional matrices of the original SDP relaxation and therefore need
not be a face obtained by classical facial reduction of that relaxation.
Let \(F_{\mathrm{fw}}\) denote the face exposed by a maximum-rank matrix
in the factor-width-\(k\) cone intersection of the first auxiliary system.  The
central comparison is
\[
P_{\mathrm{lift}}\subseteq F_{\mathrm{aff}}\subseteq F_{\mathrm{fw}}.
\]
The left inclusion shows that the polyhedral construction contains the
rank-one lift of every feasible binary point; the right inclusion shows that
the polyhedral face is contained in the face exposed by such a maximum-rank
matrix.  In fact, the containment holds for every factor-width-\(k\) exposing
matrix in the auxiliary system, not only for a maximum-rank one.

To iterate the construction, we represent each identified face
by linear equations in the original matrix variable rather than immediately
substituting a smaller matrix variable.  This representation retains the
subset indexing and allows the local RLT construction to be reapplied.  The
resulting sequence preserves all feasible binary lifts, and the factor-width
comparison holds at every iteration.

\paragraph{Contributions.}
Our contributions are: (i) an LP-based face-identification
method using a local RLT polyhedron of polynomial size for fixed \(k\); (ii) a
proof that the identified face contains every feasible binary lift and is
contained in every face exposed by a factor-width-\(k\) exposing matrix from
the auxiliary system; and (iii) an equation-based iteration that retains the
original matrix coordinates and preserves the factor-width comparison at every
step.

\paragraph{Related work.}
Classical and partial facial reduction are developed in
\cite{borwein1981regularizing,borwein1981facial,waki2013facial,permenter2018simplified}.
Factor-width cones are studied in
\cite{boman2005factor}.  Diagonally dominant sum-of-squares (DSOS) and scaled
diagonally dominant sum-of-squares (SDSOS) optimization use diagonally
dominant and scaled diagonally dominant matrices,
respectively; the latter form the factor-width-two cone
\cite{ahmadi2019dsos}.  Our comparison instead uses
factor-width cones in a facial-reduction auxiliary problem.  The
factorization of subset-indexed moment matrices and the bound-factor
inequalities are standard in RLT, Sherali--Adams, and moment descriptions
\cite{laurent2003comparison,sherali1990hierarchy,sherali2013reformulation},
whereas Affine FR uses affine-hull information to construct
faces of the positive semidefinite cone \cite{hu2024affine}.  The present results concern the
comparison with faces exposed by factor-width-\(k\) exposing matrices and its
repeated application without intermediate changes of matrix coordinates.

\paragraph{Organization.}
\Cref{sec_prel} gives the required background.
\Cref{sec:fw_lp_fr} constructs \(F_{\mathrm{aff}}\) and proves the factor-width
comparison.  \Cref{sec:structure_preserving_fr} uses face equations to
iterate the polyhedral construction in the original matrix coordinates.
\endgroup

\section{Preliminaries}\label{sec_prel}
\subsection{Notation}

For a positive integer \(p\), let \(\mathbb R^p\) be the
\(p\)-dimensional Euclidean space and let \(\mathbb S^p\) be the space of
\(p\times p\) real symmetric matrices.  We write
\(\mathbb S_+^p\) for the cone of positive semidefinite matrices; thus,
for \(X\in\mathbb S^p\), \(X\succeq0\) means
\(X\in\mathbb S_+^p\).  For a finite index set \(I\),
\(\mathbb R^I\) denotes the space of real vectors indexed by \(I\).
Vector inequalities are interpreted componentwise.

We use \(\langle\cdot,\cdot\rangle\) for the standard Euclidean inner
product on vector spaces and the Frobenius inner product on symmetric
matrix spaces.  In particular,
\(\langle X,Y\rangle:=\operatorname{tr}(XY)\) for symmetric matrices
\(X\) and \(Y\).  For \(i\in I\), \(e_i\in\mathbb R^I\) denotes the
standard unit vector indexed by \(i\); the relevant index set will be
clear from context or stated explicitly.
For a vector \(v\),
\(\operatorname{supp}(v):=\{i:v_i\neq0\}\).
For a vector \(d\), \(\operatorname{Diag}(d)\) denotes the diagonal
matrix with diagonal \(d\).  The superscript \(T\) denotes transpose.

For a matrix or linear map \(T\), \(\operatorname{range}(T)\) and
\(\ker(T)\) denote its range and kernel.  If \(T\) is a linear map
between finite-dimensional inner-product spaces, then \(T^*\) denotes
its adjoint.  For any set \(C\), \(\operatorname{span}(C)\) and
\(\aff(C)\) denote its linear span and affine hull, respectively.  When
\(C\) is convex, \(\operatorname{ri}(C)\) denotes its relative interior.
For any subset \(S\) of an inner-product space,
\[
S^\perp
:=
\{z:\langle z,s\rangle=0
\text{ for every }s\in S\}
=
\operatorname{span}(S)^\perp.
\]
For an individual element \(z\), we write
\(z^\perp:=\{z\}^\perp\).

\subsection{Facial reduction}
\label{sec:prelim_fra}
Let \(L\subseteq\mathbb S^n\) be an affine subspace such that
\(L\cap\mathbb S_+^n\neq\emptyset\).  Slater's condition holds if \(L\)
contains a positive definite matrix; in this case, the SDP system is
\emph{strictly feasible}.  If Slater's condition fails, there exists a
nonzero exposing matrix \(W\in L^\perp\cap\mathbb S_+^n\).  It exposes the
proper face
\[
F
:=
\{X\in\mathbb S_+^n:\langle W,X\rangle=0\}
=
\{X\succeq0:\range(X)\subseteq\ker(W)\},
\]
which contains \(L\cap\mathbb S_+^n\).  More generally, every
face of \(\mathbb S_+^n\) has the range-space representation
\[
F=\{X\succeq0\mid \range(X)\subseteq\mathcal{V}\}
=\{VRV^T\mid R\in\mathbb{S}^r_+\},
\]
for some subspace \(\mathcal{V}\subseteq\mathbb{R}^n\), where the columns of
\(V\in\mathbb{R}^{n\times r}\) span \(\mathcal{V}\).
Substituting \(X=VRV^T\) yields an SDP in the smaller matrix
variable \(R\in\mathbb S_+^r\).  Repeating this procedure until the reduced
system is strictly feasible yields the minimal face of \(\mathbb S_+^n\)
containing \(L\cap\mathbb S_+^n\).

\subsection{Partial facial reduction}
Partial facial reduction (partial FR) restricts the search for an exposing
matrix to an inner approximation of the positive semidefinite cone
\cite{permenter2018simplified}.  At the initial face, given a closed convex
cone \(\mathcal K\subseteq\mathbb S_+^n\), the restricted auxiliary system is
\begin{equation}\label{eq:partial_fr_search}
W\in L^\perp\cap(\mathcal K\setminus\{0\}).
\end{equation}
Every matrix satisfying this system is a valid exposing matrix.  However, the
restricted system may be infeasible even when the standard auxiliary system
has a solution, and the resulting procedure may terminate before reaching the
minimal face.

For \(k\in\{1,\ldots,n\}\), the \emph{factor-width-\(k\) cone} is
\begin{equation}\label{eq:fwk_def}
\FW_k^n
:=
\left\{
\sum_{t=1}^{q} g_tg_t^T
\;\middle|\;
q\geq1,\quad
g_t\in\mathbb R^n,\quad
|\supp(g_t)|\leq k\quad (t=1,\ldots,q)
\right\}.
\end{equation}
Thus \(\FW_k^n\subseteq\mathbb S_+^n\).  Its dual cone consists of the matrices
in \(\mathbb S^n\) whose principal submatrices of order at most \(k\) are
positive semidefinite \cite{boman2005factor}, and
\[
\FW_1^n\subseteq\FW_2^n\subseteq\cdots\subseteq\FW_n^n=\mathbb S_+^n.
\]
Choosing \(\mathcal K=\FW_k^n\) in
\eqref{eq:partial_fr_search} gives factor-width-\(k\) partial FR.  The
auxiliary problem is an LP for \(k=1\), an SOCP for \(k=2\), and an SDP with
positive semidefinite blocks of order \(k\) for \(k\geq3\).  Increasing \(k\)
enlarges the class of admissible exposing matrices, while \(k=n\) recovers the
standard auxiliary problem at the initial face.

\subsection{Binary lifts and subset-indexed moment matrices}
\label{subsec:sdp_relaxations_binary_programs}
\label{sec:lifted_feasible_set_faces}
\label{sec:local_rlt_polyhedron}

Let \(\emptyset\neq P\subseteq\{0,1\}^n\).  An affine slice
\(L\cap\mathbb{S}^{n+1}_+\) is an SDP relaxation of \(P\) if it contains the
rank-one lift of every \(x\in P\), that is,
\[
\begin{pmatrix}1\\ x\end{pmatrix}
\begin{pmatrix}1\\ x\end{pmatrix}^{T}
\in L\cap\mathbb{S}^{n+1}_+
\qquad\text{for every }x\in P.
\]
We index the coordinates of lifted vectors and the rows and columns of lifted
matrices by \(\emptyset,\allowbreak\{1\},\ldots,\allowbreak\{n\}\), where \(\emptyset\) represents
the constant monomial and \(\{i\}\) corresponds to the variable \(x_i\).

Classical facial reduction requires the exposed face to contain the entire
SDP feasible set.  Hence a nonzero exposing matrix \(W\succeq0\) must satisfy
\begin{equation}\label{eq:fr_exposing_vector_for_sdp_relaxation}
L\cap\mathbb{S}^{n+1}_+
\subseteq
\{Y\in\mathbb{S}^{n+1}_+\mid \langle W,Y\rangle=0\}.
\end{equation}
Our polyhedral construction instead targets the set of rank-one lifts
\begin{equation}\label{eq:first_level_lifted_binary_image}
P_{\mathrm{lift}}
:=
\left\{
yy^T\;\middle|\;
y=\begin{bmatrix}1\\ x\end{bmatrix},\ x\in P
\right\}
\subseteq\mathbb{S}^{n+1}_+ .
\end{equation}
Accordingly, we seek a nonzero \(\widehat W\succeq0\) such that
\begin{equation}\label{eq:first_level_face_contains_p1}
P_{\mathrm{lift}}
\subseteq
\{Y\in\mathbb{S}^{n+1}_+\mid \langle \widehat W,Y\rangle=0\}.
\end{equation}
Equivalently, \(\widehat W[1\;x^T]^T=0\) for every \(x\in P\).
The resulting face need not contain the entire SDP feasible set.  It does,
however, contain every rank-one lift of a point in \(P\).  Intersecting the SDP
relaxation with this face therefore yields a strengthening that preserves all
such rank-one lifts.  We use the term \emph{face identification} for the general
construction; it constitutes a valid facial-reduction step precisely when
\eqref{eq:fr_exposing_vector_for_sdp_relaxation} also holds with
\(W=\widehat W\).

The polyhedral construction in the next section uses the standard
subset-indexed moment-matrix notation for binary optimization.  Let
\([n]:=\{1,\ldots,n\}\).  For \(U\subseteq[n]\) and
\(0\leq j\leq |U|\), define
\[
\mathcal P_j(U):=\{I\subseteq U\mid |I|\leq j\},
\qquad
\mathcal P(U):=\mathcal P_{|U|}(U).
\]
We write \(\mathcal P_k:=\mathcal P_k([n])\).  For each \(x\in P\), its induced
moment vector \(y\in\mathbb{R}^{\mathcal{P}_k}\) is defined by
\begin{equation}\label{eq:level_k_induced_moment_vector}
y_I:=\prod_{j\in I}x_j,
\qquad I\in\mathcal{P}_k,
\end{equation}
where \(y_\emptyset=1\).

The following standard diagonalization is the key to the polyhedral
representation used in the next section
\cite[Section~3.1]{laurent2003comparison}.

\begin{lem}[Subset-indexed moment-matrix diagonalization]
\label{lem:subset_moment_diagonalization}
For \(U\subseteq[n]\) and \(y\in\mathbb R^{\mathcal P(U)}\), define
\(M_U(y):=(y_{I\cup J})_{I,J\subseteq U}\).  Let \(Z_U\) be the zeta matrix
of \(\mathcal P(U)\), with \((Z_U)_{I,J}=1\) if and only if \(I\subseteq J\).
Then
\begin{equation}
\label{eq:subset_moment_diag}
M_U(y)=Z_U\Diag(Z_U^{-1}y)Z_U^T.
\end{equation}
Consequently, \(M_U(y)\succeq0\) if and only if \(Z_U^{-1}y\geq0\).
\end{lem}

\begingroup
\section{Face identification via polyhedral relaxations and factor-width comparison}
\label{sec:fw_lp_fr}
\label{sec:direct_polyhedral_face}
\label{sec:comparison_factor_width_partial_fra}

This section compares factor-width-\(k\) partial facial reduction with face
identification via a polyhedral relaxation.  Let
\(L\cap\mathbb S_+^{n+1}\) be an SDP relaxation of the binary set \(P\), where
\(L\subseteq\mathbb S^{n+1}\) is affine.  Fix \(2\leq k\leq n\).  Every
nonzero matrix in \(L^\perp\cap\FW_k^{n+1}\) is a valid exposing matrix for
factor-width-\(k\) partial facial reduction.  For any
\(\bar W\in L^\perp\cap\FW_k^{n+1}\), including \(\bar W=0\), define
\[
F_{\mathrm{fw}}
:=
\{Y\in\mathbb S_+^{n+1}:
\range(Y)\subseteq\ker(\bar W)\}.
\]
If \(\bar W\) has maximum rank in \(L^\perp\cap\FW_k^{n+1}\), then
\(F_{\mathrm{fw}}\) is the face exposed by such a maximum-rank matrix.  The
comparison below holds for every
\(\bar W\in L^\perp\cap\FW_k^{n+1}\).

Let \(\mathcal U_k\) be the collection of all \(k\)-element subsets of
\([n]\), enumerated as \(\mathcal U_k=\{U_1,\ldots,U_q\}\), where
\(q=\binom nk\).  For
\(y\in\mathbb R^{\mathcal P_k}\), use the local matrices
\(M_{U_i}(y)\) from \Cref{lem:subset_moment_diagonalization} and define
\(M_1(y):=(y_{I\cup J})_{I,J\in\mathcal P_1}\).
Using these moment matrices, we construct a face
\(F_{\mathrm{aff}}\) satisfying
\(P_{\mathrm{lift}}\subseteq F_{\mathrm{aff}}\subseteq F_{\mathrm{fw}}\).  Define
\begin{equation}
\label{eq:comparison_affine_set_h}
H
:=
\{y\in\mathbb R^{\mathcal P_k}:M_1(y)\in L\}
\end{equation}
and
\begin{equation}
\label{eq:comparison_local_rlt_set_k}
K
:=
\{y\in\mathbb R^{\mathcal P_k}:
M_{U_i}(y)\succeq0,\ i=1,\ldots,q\}.
\end{equation}

\paragraph{Polyhedral representation.}
Although \(K\) is written using positive semidefinite constraints, it is a
polyhedron.  By \Cref{lem:subset_moment_diagonalization}, each constraint
\(M_{U_i}(y)\succeq0\) is equivalent to \(2^k\) linear inequalities.  Hence
the local moment constraints admit a description with at most
\(2^k\binom nk\) linear inequalities, a quantity polynomial in \(n\) for
fixed \(k\).  Since \(H\) is affine, \(H\cap K\) is a polyhedron.

For every \(x\in P\), its induced moment vector
\eqref{eq:level_k_induced_moment_vector} belongs to \(H\cap K\).
Let \(\Pi:\mathbb R^{\mathcal P_k}\to\mathbb R^{\mathcal P_1}\) select the
entries indexed by
\(\emptyset,\{1\},\ldots,\{n\}\).  Define
\begin{equation}
\label{eq:affine_hull_completed_subspace}
\mathcal V
:=
\Pi\operatorname{span}(\aff(H\cap K))
=
\Pi\operatorname{span}(H\cap K)
\end{equation}
and
\begin{equation}
\label{eq:direct_projected_polyhedral_subspace}
F_{\mathrm{aff}}
:=
\{Y\in\mathbb S_+^{n+1}:\range(Y)\subseteq\mathcal V\}.
\end{equation}
The affine hull of \(H\cap K\) can be recovered using standard LP methods
\cite{goldman1956theory,bertsimas1997introduction,mehdiloo2021finding}.
Its linear span, and hence the subspace in
\eqref{eq:affine_hull_completed_subspace}, then follow directly.
Consequently, constructing \(F_{\mathrm{aff}}\) does not require an auxiliary
SDP.

The following theorem compares the face obtained from the local moment
relaxation with the face exposed by \(\bar W\).

\begin{thm}[Direct factor-width comparison]
\label{thm:moment_face_inside_fw_face}
For every \(\bar W\in L^\perp\cap\FW_k^{n+1}\), the polyhedral face
\(F_{\mathrm{aff}}\) in \eqref{eq:direct_projected_polyhedral_subspace} and
the face \(F_{\mathrm{fw}}\) exposed by \(\bar W\) satisfy
\begin{equation}
\label{eq:affine_hull_completed_fw_chain}
P_{\mathrm{lift}}
\subseteq
F_{\mathrm{aff}}
\subseteq
F_{\mathrm{fw}}.
\end{equation}
\end{thm}

\begin{proof}
For the moment vector \(y\) induced by any \(x\in P\), we have
\(y\in H\cap K\) and \(\Pi y=[1\;x^T]^T\).  Hence
\(\Pi y\in\mathcal V\), so
\(\range((\Pi y)(\Pi y)^T)\subseteq\mathcal V\).  Therefore
\(P_{\mathrm{lift}}\subseteq F_{\mathrm{aff}}\).
To prove the second inclusion, fix \(y\in H\cap K\).  By the definition of
factor width, write
\begin{equation}
\label{eq:direct_fw_decomposition}
\bar W
=
\sum_{t=1}^s g_tg_t^T,
\qquad
|\supp(g_t)|\leq k.
\end{equation}
For each \(t\), choose a set \(U_t\subseteq[n]\) of cardinality \(k\)
containing every \(j\) such that \(\{j\}\in\supp(g_t)\).  Extend \(g_t\) to
\(\widetilde g_t\in\mathbb R^{\mathcal P(U_t)}\) by retaining its entries on
\(\mathcal P_1(U_t)\) and setting all higher-order entries to zero.  The
principal submatrix of \(M_{U_t}(y)\) indexed by
\(\mathcal P_1(U_t)\) is exactly the principal submatrix of \(M_1(y)\) with
the same index set.  Since \(\widetilde g_t\) is zero outside
\(\mathcal P_1(U_t)\), it follows that
\begin{equation}
\label{eq:direct_local_quadratic_identity}
\widetilde g_t^TM_{U_t}(y)\widetilde g_t
=
g_t^TM_1(y)g_t.
\end{equation}

Since \(y\in H\cap K\), we have \(M_1(y)\in L\), and therefore
\(\langle \bar W,M_1(y)\rangle=0\).
Every matrix \(M_{U_t}(y)\) is positive semidefinite.  Hence
\begin{align*}
0
&=
\langle \bar W,M_1(y)\rangle\\
&=
\sum_{t=1}^s g_t^TM_1(y)g_t\\
&=
\sum_{t=1}^s
\widetilde g_t^TM_{U_t}(y)\widetilde g_t
\end{align*}
is a sum of nonnegative terms.  Every term is therefore zero.  Since
\(M_{U_t}(y)\succeq0\), this implies
\(M_{U_t}(y)\widetilde g_t=0\) for \(t=1,\ldots,s\).
The row indexed by \(\emptyset\) gives
\[
0
=
(g_t)_\emptyset y_\emptyset
+
\sum_{j\in U_t}(g_t)_{\{j\}} y_{\{j\}}
=
g_t^T\Pi y,
\qquad t=1,\ldots,s,
\]
where the last equality uses the choice of \(U_t\).
It follows from \eqref{eq:direct_fw_decomposition} that
\(\bar W\Pi y=0\).  Thus
\(\Pi y\in\ker(\bar W)\) for every \(y\in H\cap K\), and taking spans
gives \(\mathcal V\subseteq\ker(\bar W)\).  This proves the second inclusion
in \eqref{eq:affine_hull_completed_fw_chain}.
\end{proof}

The comparison above concerns faces containing \(P_{\mathrm{lift}}\).  For
\(F_{\mathrm{aff}}\) to contain the feasible set of the original SDP, as
required by classical facial reduction, one would additionally need
\(\range(Y)\subseteq\mathcal V\) for every
\(Y\in L\cap\mathbb S_+^{n+1}\).  This condition is not assumed or established
by the construction.

\begin{ex}[A simple illustration of strict containment]
\label{ex:strict_affine_factor_width_containment}
Let \(n=k=2\), let
\[
P=\{(0,0),(1,1)\},
\]
and consider the Shor SDP relaxation \(L\cap\mathbb S_{+}^{3}\) defined by
the affine subspace
\[
L
:=
\{Y\in\mathbb S^3:
Y_{\emptyset,\emptyset}=1,\
Y_{\{1\},\{1\}}=Y_{\emptyset,\{1\}},\
Y_{\{2\},\{2\}}=Y_{\emptyset,\{2\}},\
Y_{\emptyset,\{1\}}=Y_{\emptyset,\{2\}}\}.
\]
The rank-one lift of each point in \(P\) belongs to
\(L\cap\mathbb S_+^3\).  Moreover,
\[
Y^*
:=
\begin{bmatrix}
1&1/2&1/2\\
1/2&1/2&1/4\\
1/2&1/4&1/2
\end{bmatrix}
\in L
\]
is positive definite, since its Schur complement with respect to the
upper-left entry is \(\tfrac14 I_2\).  Hence the SDP relaxation satisfies
Slater's condition.  It follows that
\(L^\perp\cap\FW_2^3=\{0\}\).  Thus \(\bar W=0\) and
\(F_{\mathrm{fw}}=\mathbb S_+^3\).

For the polyhedral construction,
\[
H
=
\{y\in\mathbb R^{\mathcal P_2}:
y_\emptyset=1,\ y_{\{1\}}=y_{\{2\}}\}.
\]
Thus \(y_{\{1,2\}}\) is unrestricted in \(H\) and is constrained only
through \(K\).
Every vector in \(\Pi(H\cap K)\) therefore satisfies
\(v_{\{1\}}=v_{\{2\}}\).  Conversely, the moment vectors induced by the two
points in \(P\) belong to \(H\cap K\), and their first-level projections are
\((1,0,0)^T\) and \((1,1,1)^T\).  Consequently,
\[
\mathcal V
=
\{v\in\mathbb R^{\mathcal P_1}:v_{\{1\}}=v_{\{2\}}\},
\]
and \(F_{\mathrm{aff}}\) is the proper face associated with this subspace.
Thus
\[
F_{\mathrm{aff}}
\subsetneq
F_{\mathrm{fw}}
=
\mathbb S_+^3.
\]
This example shows that the polyhedral construction can identify a proper
face containing \(P_{\mathrm{lift}}\) even when the original SDP relaxation
is strictly feasible, in which case the factor-width auxiliary system admits
no nonzero exposing matrix.
\end{ex}

\section{Face identification in the original matrix coordinates}
\label{sec:structure_preserving_fr}

To iterate the polyhedral construction, we impose each
identified face by linear equations in the original matrix variable, thereby
retaining the subset indexing.  We first establish this representation for a
face containing the SDP feasible set and then use the same face-equation
representation for the
strengthened relaxations obtained from the \(P_{\mathrm{lift}}\)-preserving
faces identified by the polyhedral construction.

Related changes of SDP representation appear in the work of Pataki, who uses
row operations and congruence transformations to obtain canonical forms
\cite{pataki2017bad}, and in RiNNAL+, where an enlarged doubly nonnegative
relaxation is related to a lower-dimensional SDP--RLT relaxation
\cite{hou2026rinnalplus}.  Here, the face-equation representation enables
repeated polyhedral face identification in the original moment-matrix
coordinates.

\subsection{Structure-preserving facial reduction}
\label{sec:structure_preserving_fr_clean}

Let \(\A:\mathbb S^N\to\mathbb R^m\) be linear, let
\(b\in\mathbb R^m\), and let \(\A^*\) denote the adjoint of \(\A\).  Consider
the SDP feasible set
\begin{equation}
\label{eq:face_equation_original_sdp}
L\cap\mathbb S_+^N,
\qquad
L:=\{Y\in\mathbb S^N:\A(Y)=b\}.
\end{equation}
Suppose a face \(F\) contains this feasible set.  Let the columns of
\(V\in\mathbb R^{N\times r}\) form an orthonormal basis of the subspace
defining \(F\), and complete \(V\) to an orthogonal matrix \([\,V\ U\,]\).  Then
\[
F
=
\{Y\succeq0:\range(Y)\subseteq\range(V)\}
=
\{VRV^T:R\succeq0\}.
\]
Instead of substituting \(Y=VRV^T\), observe that, for a symmetric matrix
\(Y\),
\begin{equation}
\label{eq:face_equations_simple}
Y\in\operatorname{span}(F)
\quad\Longleftrightarrow\quad
YU=0.
\end{equation}
Define
\begin{equation}
\label{eq:face_equation_affine_space}
L_F:=
L\cap\operatorname{span}(F)
=
\{Y\in L:YU=0\}.
\end{equation}
Because \(F\) contains the feasible set in
\eqref{eq:face_equation_original_sdp},
\begin{equation}
\label{eq:face_equation_exact_feasible_set}
L\cap\mathbb S_+^N
=
L_F\cap\mathbb S_+^N.
\end{equation}
Thus \(L_F\) gives an equivalent reformulation that retains the original
matrix variable and its indexing.

Let \(\mathcal B:\mathbb S^N\to\mathbb R^{N\times(N-r)}\) be defined by
\(\mathcal B(Y):=YU\), with both spaces equipped with the Frobenius inner
product.  Its adjoint is
\[
\mathcal B^*(Z)=\tfrac12(ZU^T+UZ^T),
\qquad Z\in\mathbb R^{N\times(N-r)}.
\]
From \eqref{eq:face_equations_simple} and the
standard relation \(\range(\mathcal B^*)=(\ker\mathcal B)^\perp\),
\begin{equation}
\label{eq:face_equation_adjoint_range}
\ker(\mathcal B)=\operatorname{span}(F),
\qquad
\range(\mathcal B^*)=F^\perp,
\end{equation}
where \(F^\perp:=\operatorname{span}(F)^\perp\).

We compare the exposing vectors obtained from the two representations
\(L\cap F\) and \(L_F\cap\mathbb S_+^N\).  In the first representation, an
exposing vector has the form \(\A^*(\lambda)\), where
\(\lambda\in\mathbb R^m\).  In the second, the equations \(YU=0\) introduce
a matrix multiplier \(Z\in\mathbb R^{N\times(N-r)}\), and the exposing vector
has the form \(\A^*(\lambda)+\mathcal B^*(Z)\).  These two matrices differ by
an element of \(F^\perp\); hence, they have the same inner product with every
matrix in \(F\) and expose the same face of \(F\).

\begin{lem}[Equivalence of auxiliary certificates]
\label{lem:second_step_auxiliary_equivalence}
The multipliers \(\lambda\in\mathbb R^m\) satisfying
\begin{equation}
\label{eq:second_step_current_face_auxiliary}
V^T\A^*(\lambda)V\in\mathbb S_+^r\setminus\{0\},
\qquad
b^T\lambda=0,
\end{equation}
are exactly those for which there exists
\(Z\in\mathbb R^{N\times(N-r)}\) such that \((\lambda,Z)\) satisfies
\begin{equation}
\label{eq:second_step_ambient_auxiliary}
\A^*(\lambda)+\mathcal B^*(Z)\succeq0,
\qquad
\A^*(\lambda)+\mathcal B^*(Z)\notin F^\perp,
\qquad
b^T\lambda=0.
\end{equation}
\end{lem}

\begin{proof}
For every \(S\in\mathbb S^N\),
\begin{equation}
\label{eq:face_dual_ambient_completion}
V^TSV\succeq0
\quad\Longleftrightarrow\quad
S+D\succeq0
\quad\text{for some }D\in F^\perp.
\end{equation}
Indeed, since \(F=\{VRV^T:R\succeq0\}\), let \(P_V:=VV^T\).  If
\(V^TSV\succeq0\),
define \(D:=P_VSP_V-S\).
Then \(V^TDV=0\), so \(D\in F^\perp\), and
\[
S+D=P_VSP_V=V(V^TSV)V^T\succeq0.
\]
Conversely, if \(S+D\succeq0\) for some \(D\in F^\perp\), then
\(V^TSV=V^T(S+D)V\succeq0\), where \(V^TDV=0\) follows from
\(D\in F^\perp\).  Combining
\eqref{eq:face_dual_ambient_completion} with
\(\range(\mathcal B^*)=F^\perp\) proves the equivalence of the cone
conditions.  Since \(F^\perp\) is a subspace, adding
\(\mathcal B^*(Z)\in F^\perp\) does not change membership in \(F^\perp\).
Moreover, \(V^TSV=0\) if and only if \(S\in F^\perp\).
Hence the two nonzero conditions are equivalent, and the equation
\(b^T\lambda=0\) is unchanged.
\end{proof}

To describe the second facial-reduction step in the original matrix
coordinates, suppose that \(F\) is the face obtained at the first step.
Because \(F\) contains the original SDP feasible set,
\eqref{eq:face_equation_exact_feasible_set} applies.
By \Cref{lem:second_step_auxiliary_equivalence}, the auxiliary system based on
\(L_F\cap\mathbb S_+^N\) has exactly the same feasible \(\lambda\)-multipliers
as the standard auxiliary system over \(F\), and the corresponding exposing
matrices expose the same face of \(F\).
Thus, the second step can be performed in the original matrix coordinates.
The same representation can be used at subsequent steps whenever the newly
obtained face contains the current SDP feasible set.

\subsection{Iterative face identification using polyhedral relaxations}
\label{sec:repeated_factor_width_comparison}

We now iterate the polyhedral construction.  We first describe
the second application and then state the full recursion.  The first
application in \Cref{sec:comparison_factor_width_partial_fra} produces
\(F_{\mathrm{aff}}\subseteq F_{\mathrm{fw}}\).
Unlike the face \(F\) in
\Cref{sec:structure_preserving_fr_clean}, \(F_{\mathrm{aff}}\) need not contain
the feasible set of the original SDP relaxation.  After the first
face-identification step, we therefore consider the strengthened relaxation
\(L\cap F_{\mathrm{aff}}\), which contains every feasible binary lift.

Let the columns of \(U\) span the orthogonal complement of the subspace
defining \(F_{\mathrm{aff}}\), and define
\(\widetilde L
=L\cap\operatorname{span}(F_{\mathrm{aff}})\).  Equivalently,
\begin{equation}
\label{eq:iterated_equation_augmented_space}
\widetilde L:=\{Y\in L:YU=0\},
\qquad
\widetilde H:=\{y\in\mathbb R^{\mathcal P_k}:M_1(y)\in\widetilde L\},
\qquad
\widetilde Q:=\widetilde H\cap K.
\end{equation}
Since
\(F_{\mathrm{aff}}
=\operatorname{span}(F_{\mathrm{aff}})\cap\mathbb S_+^{n+1}\), we have
\begin{equation}
\label{eq:second_step_strengthened_relaxation}
\widetilde L\cap\mathbb S_+^{n+1}
=
L\cap F_{\mathrm{aff}}.
\end{equation}
Thus, the face equations represent the strengthened relaxation exactly in the
original matrix coordinates.
From this polyhedron, construct
\begin{equation}
\label{eq:second_step_polyhedral_face}
\widetilde{\mathcal V}:=\Pi\operatorname{span}(\widetilde Q),
\qquad
\widetilde F_{\mathrm{aff}}:=
\{Y\succeq0:\range(Y)\subseteq\widetilde{\mathcal V}\}.
\end{equation}
For any \(\widetilde W\in\widetilde L^\perp\cap\FW_k^{n+1}\), define
\[
\widetilde F_{\mathrm{fw}}
:=
\{Y\in\mathbb S_+^{n+1}:\range(Y)\subseteq\ker(\widetilde W)\}.
\]

\begin{prop}[Second-step comparison]
\label{prop:p1_preserving_face_update}
For every \(\widetilde W\in\widetilde L^\perp\cap\FW_k^{n+1}\), the
second-step faces satisfy
\begin{equation}
\label{eq:p1_preserving_face_update}
P_{\mathrm{lift}}\subseteq \widetilde F_{\mathrm{aff}}
\subseteq\widetilde F_{\mathrm{fw}}.
\end{equation}
\end{prop}

\begin{proof}
Since \(P_{\mathrm{lift}}\subseteq F_{\mathrm{aff}}\), imposing \(YU=0\)
preserves every feasible binary lift.  Hence
\(\widetilde L\cap\mathbb S_+^{n+1}\) is an SDP relaxation of the same binary
set.  Applying \Cref{thm:moment_face_inside_fw_face} with \(L\) replaced by
\(\widetilde L\) proves \eqref{eq:p1_preserving_face_update}.
\end{proof}

Repeating the argument yields the full sequence.  Initialize
\(F_{\mathrm{aff}}^{(0)}:=\mathbb S_+^{n+1}\).  If the columns of
\(U^{(i-1)}\) span the orthogonal complement of the subspace defining
\(F_{\mathrm{aff}}^{(i-1)}\), define
\begin{equation}
\label{eq:iterated_polyhedral_relaxation}
\begin{aligned}
L^{(i-1)}&:=\{Y\in L:YU^{(i-1)}=0\},\\
H^{(i-1)}&:=\{y\in\mathbb R^{\mathcal P_k}:M_1(y)\in L^{(i-1)}\},
\qquad Q^{(i-1)}:=H^{(i-1)}\cap K,\\
F_{\mathrm{aff}}^{(i)}
&:=\{Y\succeq0:\range(Y)\subseteq
\Pi\operatorname{span}(Q^{(i-1)})\}.
\end{aligned}
\end{equation}
Here \(U^{(0)}\) has no columns and \(L^{(0)}=L\).  For each \(i\geq1\), choose
an arbitrary matrix
\(\bar W^{(i)}\in(L^{(i-1)})^\perp\cap\FW_k^{n+1}\) and define
\[
F_{\mathrm{fw}}^{(i)}
:=
\{Y\in\mathbb S_+^{n+1}:\range(Y)\subseteq\ker(\bar W^{(i)})\}.
\]

\begin{cor}[Repeated comparison in the original matrix coordinates]
\label{cor:repeated_lp_face_comparison}
For every \(i\geq1\) and every such choice of \(\bar W^{(i)}\),
\begin{equation}
\label{eq:iterated_face_chain}
P_{\mathrm{lift}}
\subseteq
F_{\mathrm{aff}}^{(i)}
\subseteq
F_{\mathrm{fw}}^{(i)}.
\end{equation}
\end{cor}

\begin{proof}
The result follows by induction.  The case \(i=1\) is
\Cref{thm:moment_face_inside_fw_face}.  If
\(P_{\mathrm{lift}}\subseteq F_{\mathrm{aff}}^{(i-1)}\), then the equations
defining \(L^{(i-1)}\) preserve every feasible binary lift.  Hence
\(L^{(i-1)}\cap\mathbb S_+^{n+1}\) is an SDP relaxation of the same binary set,
and another application of \Cref{thm:moment_face_inside_fw_face} proves
\eqref{eq:iterated_face_chain}.
\end{proof}

In particular, \Cref{prop:p1_preserving_face_update,cor:repeated_lp_face_comparison}
apply when the factor-width matrices are chosen to have maximum rank in their
respective cone intersections.

The polyhedral faces are nested.  Indeed, if \(y\in Q^{(i-1)}\), then
\(M_1(y)U^{(i-1)}=0\), so
\(\Pi y=M_1(y)e_{\emptyset}\) belongs to the subspace defining
\(F_{\mathrm{aff}}^{(i-1)}\).  Therefore
\(F_{\mathrm{aff}}^{(i)}\subseteq F_{\mathrm{aff}}^{(i-1)}\).
Consequently, \(L^{(i)}\subseteq L^{(i-1)}\) and
\(Q^{(i)}\subseteq Q^{(i-1)}\).

A strict inclusion can occur because constructing
\(F_{\mathrm{aff}}^{(i)}\) uses only the projected condition
\(\Pi y\in\Pi\operatorname{span}(Q^{(i-1)})\), whereas the next polyhedron
imposes the full matrix equation \(M_1(y)U^{(i)}=0\).  Consequently,
\(Q^{(i)}\) can be a proper subset of \(Q^{(i-1)}\).

The sequence in \eqref{eq:iterated_face_chain} defines iterative face
identification using polyhedral relaxations.  It is not, in general, a
classical facial-reduction sequence because
\(F_{\mathrm{aff}}^{(i)}\) need not contain the current SDP feasible set
\(L^{(i-1)}\cap\mathbb S_+^{n+1}\).  When this containment holds, imposing the
corresponding face equations preserves the full SDP feasible set.  If, in
addition, each face is obtained from a valid facial-reduction certificate for
the current affine system, the updates constitute the structure-preserving
implementation of standard FR described in
\Cref{sec:structure_preserving_fr_clean}.

We say that the procedure terminates after \(r\) steps if the next update does
not produce a strictly smaller identified face, that is, if
\(F_{\mathrm{aff}}^{(r+1)}=F_{\mathrm{aff}}^{(r)}\).  Upon termination, let
the columns of \(V^{(r)}\) span the subspace defining
\(F_{\mathrm{aff}}^{(r)}\).  The strengthened SDP on this face is obtained by
substituting the following expression into the original SDP formulation:
\begin{equation}
\label{eq:structure_preserving_final_reformulation}
Y=V^{(r)}R(V^{(r)})^T,
\qquad R\succeq0.
\end{equation}
This formulation contains every matrix in \(P_{\mathrm{lift}}\), but it need
not be equivalent to the original SDP relaxation.  Termination alone does not imply
strict feasibility.  If
\(F_{\mathrm{aff}}^{(r)}\) is the minimal face containing
\(L\cap\mathbb S_+^{n+1}\), then the formulation in \(R\) is equivalent to the
original SDP and satisfies Slater's condition.

\endgroup

\section{Conclusion}

\begingroup
For SDP relaxations of binary programs, we constructed a
polyhedron from local RLT bound-factor inequalities.  For fixed \(k\), this
polyhedron has polynomial size, and its projected span
determines a face of the positive semidefinite cone containing every feasible
binary lift.

For \(k\geq3\), the factor-width-\(k\) auxiliary problem is itself an SDP with
positive semidefinite blocks of order \(k\).  In contrast, for fixed \(k\), the
proposed face can be identified using LP.  For the same SDP formulation and the
same value of \(k\), this face is contained in every face exposed by a
factor-width-\(k\) exposing matrix from the auxiliary system.  Because it may
exclude fractional SDP-feasible matrices, imposing the face can strengthen the
relaxation while preserving every feasible binary lift.

Representing each identified face by linear equations enables
repeated application of the construction while retaining the original
moment-matrix indexing.  Every iteration preserves the feasible binary lifts
and satisfies the same factor-width comparison.  When the identified faces
also contain the full SDP feasible set and arise from valid FR certificates,
the iteration specializes to standard facial reduction.
\endgroup

\section*{Acknowledgments}
The work of all authors was supported by the Air Force Office of Scientific
Research under award number FA9550-23-1-0508.

\bibliographystyle{siam}
\bibliography{mybib}
\addcontentsline{toc}{section}{Bibliography}

\end{document}